\documentclass[12pt]{article}
\usepackage{amsmath,amssymb,amsthm}
\usepackage{enumerate}
\usepackage{graphicx}
\usepackage[T1]{fontenc}
\usepackage{caption}
\usepackage{hyperref}
\usepackage{cite}

\def\draw #1 by #2 (#3){
	\vbox to #2{
		\hrule width #1 height 0pt depth 0pt
		\vfill
		\special{picture #3} 
	}
}

\def\scaleddraw #1 by #2 (#3 scaled #4){{
		\dimen0=#1 \dimen1=#2
		\divide\dimen0 by 1000 \multiply\dimen0 by #4
		\divide\dimen1 by 1000 \multiply\dimen1 by #4
		\draw \dimen0 by \dimen1 (#3 scaled #4)}
}

\newtheorem{theorem}{Theorem}[section]
\newtheorem{example}[theorem]{Example}
\newtheorem{problem}[theorem]{Problem}
\newtheorem{defin}[theorem]{Definition}

\newtheorem{corollary}[theorem]{Corollary}
\newtheorem{remark}[theorem]{Remark}
\usepackage{xcolor}
\usepackage{float}
\usepackage{graphicx,subcaption,lipsum}
\usepackage{hyperref}
\newtheorem{nt}{Note}

\newcommand{\singlespacing}{\let\CS=\@currsize\renewcommand{\baselinestretch}{1}\tiny\CS}
\newcommand{\oneandahalfspacing}{\let\CS=\@currsize\renewcommand{\baselinestretch}{1.25}\tiny\CS}
\newcommand{\doublespacing}{\let\CS=\@currsize\renewcommand{\baselinestretch}{1.35}\tiny\CS}

\newtheorem{proposition}[theorem]{Proposition}

\newtheorem{definition}{Definition}

\newtheorem{rule-def}[theorem]{Rule}

\begin{document}
	\baselineskip 16pt
	\newcommand{\la}{\lambda}
	\newcommand{\si}{\sigma}
	\newcommand{\ol}{1-\lambda}
	\newcommand{\be}{\begin{equation}}
		\newcommand{\ee}{\end{equation}}
	\newcommand{\bea}{\begin{eqnarray}}
		\newcommand{\eea}{\end{eqnarray}}

	\baselineskip=0.30in
	\begin{center}
		{\Large  \textbf{Secure coalitions in graphs} }\\
		\vspace*{0.3cm} 
	\end{center}
    \begin{center}
		Swathi Shetty$^{1}$, Sayinath Udupa N. V.$^{*,2}$, B. R. Rakshith$^{3}$\\
		Manipal Institute of Technology\\ Manipal Academy of Higher Education\\ Manipal, India. \\
		swathi.dscmpl2022@learner.manipal.edu$^{1}$\\
        sayinath.udupa@manipal.edu	$^{*,2}$ (Corresponding author)\\
		ranmsc08@yahoo.co.in; rakshith.br@manipal.edu$^{3}$.
	\end{center}
    \begin{abstract}
    A secure coalition in a graph $G$ consists of two disjoint vertex sets $V_1$ and $V_2$, neither of which is a secure dominating set, but whose union $V_1 \cup V_2$ forms a secure dominating set. A secure coalition partition ($sec$-partition) of $G$ is a vertex partition $\pi = \{V_1, V_2, \dots, V_k\}$ where each set $V_i$ is either a secure dominating set consisting of a single vertex of degree $n-1$, or a set that is not a secure dominating set but forms a secure coalition with some other set $V_j \in \pi$. The maximum cardinality of a secure coalition partition of $G$ is called the secure coalition number of $G$, denoted $SEC(G)$. For every $sec$-partition $\pi$ of a graph $G$, we associate a graph called the secure coalition graph of $G$ with respect to $\pi$, denoted $SCG(G,\pi)$, where the vertices of $SCG(G,\pi)$ correspond to the sets $V_1, V_2, \dots, V_k$ of $\pi$, and two vertices are adjacent in $SCG(G,\pi)$ if and only if their corresponding sets in $\pi$ form a secure coalition in $G$. In this study, we prove that every graph admits a $sec$-partition. Further, we characterize the graphs $G$ with $SEC(G) \in \{1,2,n\}$ and all trees $T$ with $SEC(T) = n-1$. Finally, we show that every graph $G$ without isolated vertices is a secure coalition graph.
\end{abstract}
  \section{Introduction}
    Let $G=(V, E)$ be a simple graph with a non-empty set of vertices $V$ and the edge set $E$, of order $n$ and size $m$. If the vertex $u$ is adjacent to $v$, we write $u\sim v$. If not, then $u\not\sim v$. The set of vertices that are adjacent to the vertex $v$ is called the open neighborhood of $v$, denoted as $N(v)$, while its closed neighborhood is the set $N[v]=N(v)\cup \left\{v\right\}$. The degree of vertex $v$, represented by $d(v)=|N(v)|$. We use $\delta(G)[\Delta(G)]$ to denote the minimum [maximum] degree of $G$. Throughout this paper, a vertex of degree $1$ is called a pendant vertex, and its unique neighbor is called a support vertex. A vertex of degree $n-1$ is called a full vertex. For a subset $U\subseteq V(G)$, the subgraph induced by $U$ is denoted by $G[U]$.\par
    A $\mathcal{P}$-coalition partition~\cite{haynes2023coalition} is a vertex partition $\pi = \left\{V_1, V_2, \dots, V_k\right\}$ such that no set $V_i$ is a $\mathcal{P}$-set, but for every $V_i$ there exists a $V_j$ such that $V_i \cup V_j$ is a $\mathcal{P}$-set.
     A set $D \subseteq V(G)$ is a \textit{dominating set} of a graph $G$ if every vertex in $V(G) \backslash D$ is adjacent to at least one vertex in $D$.  For more details refer~\cite{haynes2013fundamentals,haynes2020topics}.
    The concept of coalitions in graphs was introduced by Haynes et.al. in 2020~\cite{haynes2020introduction} by considering the property $\mathcal{P}$ as domination.
     A \textit{coalition partition} ($c$-partition), in a graph $G$ is a vertex partition $\pi=\left\{V_1, V_2,\dots, V_k\right\}$ such that every set $V_i$ of $\pi$ is either a singleton dominating set, or is not a dominating set but forms a coalition with another set $V_j$ in $\pi$.
    The \textit{coalition number} $C(G)$ equals the maximum order $k$ of a $c$-partition of $G$, and a $c$-partition of $G$ having order $C(G)$ is called a \textit{$C(G)$-partition}. For a graph $G$ with vertex set $V=\left\{v_1,v_2,\dots,v_n\right\}$, the partition $\pi_1=\left\{V_1,V_2,\dots,V_n\right\}$, where $1\le i\le n$ known as the \textit{singleton partition}.\par
    In~\cite{haynes2020introduction}, the authors established upper and lower bounds for the coalition number and determined exact values for paths and cycles. This line of investigation was extended in~\cite{haynes2021upper}, where additional upper bounds were derived in terms of the minimum and maximum degree of the graph. The notion of coalition graphs was formally introduced in~\cite{haynes2023coalition}, with the key result that every graph can be viewed as a coalition graph. Building upon this,~\cite{haynes2023coalition1} offered a complete characterization of coalition graphs for trees, paths, and cycles.  In a related direction, Bakhshesh et al.~\cite{bakhshesh2023coalition} characterized graphs $G$ of order $n$ with $\delta(G) = 0$ and $\delta(G) = 1$ for which $C(G) = n$, and also characterized trees $T$ satisfying $C(T) = n - 1$. In~\cite{haynes2023self}, the authors provided a complete characterization of self-coalition graphs. The Different variations of coalition introduced and discussed in~\cite{alikhani2024total,alikhani2025independent,henning2025double,samadzadeh2025paired,bakhshesh2025minmin}.
    \section{Secure coalition in graphs}
    A \textit{secure dominating set} $S$ of a graph $G$ is a dominating set with the property that each vertex $u\in V\backslash S$ is adjacent to a vertex $v\in S$ such that $\left(S\backslash\left\{v\right\}\right)\cup \left\{u\right\}$ is a dominating set. The minimum cardinality of a secure dominating set is called the secure domination number, denoted by $\gamma_s(G)$~\cite{cockayne2005protection}. For the $\mathcal{P}$-coalitions in this paper, we consider the property of being a secure dominating set. We first define a secure coalition of a graph $G$.
\begin{definition}
A \textit{secure-coalition} in a graph $G$ consists of two disjoint sets of vertices $V_1$ and $V_2$, neither of which is a secure dominating set, but whose union $V_1 \cup V_2$ is a secure dominating set. We say that the sets $V_1$ and $V_2$ form a secure-coalition.\end{definition}
    \begin{definition}A \textit{secure coalition partition} ($sec$-partition) of $G$ is a vertex partition $\pi=\left\{V_1, V_2,\dots, V_k\right\}$ such that every set $V_i$ of $\pi$ is either a secure dominating set consisting of a single vertex of degree $n-1$, or is not a secure dominating set but it forms a secure coalition with another set $V_j\in \Pi$. The maximum cardinality of a secure coalition partition of $G$ is called \textit{secure coalition number} of $G$, denoted by $SEC(G)$.
    \end{definition}
    \begin{remark} A coalition partition of a graph $G$ need not be a secure coalition partition. 
    However, every secure coalition partition is a coalition partition provided no vertex subset in the partition is a dominating set. For instance, in the path $P_6=(v_1,v_2,v_3,v_4,v_5,v_6)$, 
    the partition $\pi_1=\{\{v_1,v_6\},\{v_2\},\{v_3\},\{v_4\},\{v_5\}\}$  is a coalition partition but not a secure one, since $\{v_2\}\cup\{v_5\}$ is dominating but not securely dominating. Conversely, a secure coalition partition $\pi_2=\left\{\left\{v_2,v_5\right\},\left\{v_1,v_6\right\},\left\{v_4\right\},\left\{v_3\right\}\right\}$ of $P_6$ 
    fails to be a coalition partition since $\{v_2,v_5\}$ is dominating. Thus, for any graph $G$, $SEC(G) \le C(G)$,
    where $SEC(G)$ and $C(G)$ denote the secure coalition number and coalition number of $G$, respectively.
\end{remark}
   The following theorem establishes the existence of a secure coalition partition for every graph.
 \begin{theorem}\label{mindelta2}
 Every graph $G$ has a secure coalition partition.
    \end{theorem}
    \begin{proof}
    Let $G \not\cong K_n$ be a graph with no isolated vertices. Choose a vertex $v_1 \in V(G)$ such that $d(v_1)=\delta(G)$ and let $N(v_1)=\left\{v_2,\ldots,v_{\delta+1}\right\}$. Since $G\not\cong K_n$, it follows that $V(G)\setminus N[v_1]\neq \varnothing$.
Define $W=V(G)\setminus N[v_1]$. For each integer $i$ with $1\le i\le \delta+1$, let $V_i=\left\{v_i\right\}$. Then $\pi=\left\{V_1,V_2,\ldots,V_{\delta+1},W\right\}$ is a partition of $V(G)$.
We now show that $\pi$ is an $sec$-partition.\par
Since $G\not\cong K_n$ and each $V_i$ ($1\le i\le \delta+1$) is a singleton set, none of the sets $V_i$ is a secure dominating set of $G$. Moreover, since $N[v_1]\cap W=\varnothing$, the set $W$ does not dominate $v_1$, and hence $W$ is not a secure dominating set. We now show that, for each $i$ with $1\le i\le \delta+1$, the set $S=V_i\cup W$ is a secure dominating set of $G$.\\
\textbf{Case 1.} $S=V_i\cup W$, where $2\le i\le \delta+1$.\\
\textbf{Subcase 1.1.} Every vertex $v_i$, $i\neq 1$, has a neighbor in $W$.\\
Since each vertex $v_i$, $i\neq 1$, has a neighbor in $W$, every such vertex is dominated by a vertex of $W$. Moreover, the vertex $v_1$ is dominated by $v_i$. Hence, $S$ is a dominating set of $G$.\\
To show that $S$ is secure dominating set, let $x\in V(G)\setminus S$. Then $x\in N[v_1]\setminus\{v_i\}$.\\ If $x=v_1$, consider the set $(S\setminus\{v_i\})\cup\{v_1\}$.
In this set, every vertex of $N[v_1]$ is dominated by $v_1$, while remaining vertices dominated by vertices of $W$. Therefore, $(S\setminus\{v_i\})\cup\{v_1\}$ is a dominating set. Now suppose that $x\in N(v_1)\setminus\{v_i\}$. Then consider the set
$(S\setminus\{v_i\})\cup\{x\}$. Since every vertex in $N(v_1)$ has a neighbor in $W$, all vertices of $N(v_1)$ remain dominated by vertices of $W$. Furthermore, $v_1$ is dominated by $x$. Thus,
$(S\setminus\{v_i\})\cup\{x\}$ is a dominating set. Therefore,  $S$ is a secure dominating set of $G$.\\[2mm]
\textbf{Subcase 1.2.} None of the vertices $v_i$, $i\neq 1$, has a neighbor in $W$.\\
This subcase applies only to graphs with no full vertex. Since
$d(v_i)\ge d(v_1)=\delta(G)$ for each $i\neq 1$, and none of the vertices $v_i$ is adjacent to any vertex in $W$, it follows that every $v_i$ must be adjacent to all vertices in $N(v_1)$. Consequently, each $v_i$ dominates every vertex in $N[v_1]$. Since the vertices in $W$ dominate themselves, $S=V_i\cup W$ is a dominating set of $G$.\\
Now let $x\in V(G)\setminus S$. Then $x\in N[v_1]\setminus\{v_i\}$. Since every vertex in $N[v_1]$ dominates all vertices of $N[v_1]$, the set $(S\setminus\{v_i\})\cup\{x\}$
remains a dominating set of $G$. Therefore, $S$ is a secure dominating set of $G$.\\[2mm]
\textbf{Subcase 1.3.} There exists at least one vertex $v_i$ such that $N(v_i)\cap W\neq\emptyset$, and at least one vertex $v_j$ such that $N(v_j)\cap W=\emptyset$.\\
Combining the arguments used in Subcases 1.1 and 1.2, it follows that \(S=V_i\cup W\) is a dominating set of $G$. Moreover, for every vertex $x\in V(G)\setminus S$, the set
$(S\setminus\{v_i\})\cup\{x\}$ is also a dominating set of $G$. Consequently, $S$ is a secure dominating set of $G$.\\[2mm]
\textbf{Case 2.} $S=\{v_1\}\cup W$.\\
Since $v_1$ dominates every vertex in $N[v_1]$, and each vertex in
$V(G)\setminus N[v_1]=W$ is dominated by itself, it follows that $S$ is a dominating set of $G$.\\
To show that $S$ is secure dominating set, let $x\in V(G)\setminus S$. Then $x\in N(v_1)$. Consider the set
$(S\setminus\{v_1\})\cup\{x\}$. The vertex $v_1$ is dominated by $x$, and every vertex in $N(v_1)$ is dominated either by $x$ or by a vertex of $W$. Moreover, each vertex in $W$ dominated by itself. Therefore,
$(S\setminus\{v_1\})\cup\{x\}$ is a dominating set of $G$ for every $x\in N(v_1)$. Therefore, $S$ is a secure dominating set of $G$.\\
By Cases 1 and 2, it follows that for each $i$, $1\le i\le \delta+1$, the set $S=V_i\cup W$ is a secure dominating set of $G$.\par
Let $G\ncong \overline{K_n}$ be a graph with a nonempty set $I$ of isolated vertices, and let $G'=G[V(G)\setminus I]$. Let $y\in V(G')$ be a vertex with $d_{G'}(y)=\delta(G')$, and let
$N_{G'}(y)=\{y_1,y_2,\ldots,y_l\}$, where $l=\delta(G')$. Define $V_i=\{y_i\}$, $1\le i\le l$, and $V_{l+1}=\{y\}$. Further, let $W=\bigl(V(G')\setminus N_{G'}[y]\bigr)\cup I$.
Then $\pi=\{V_1,V_2,\ldots,V_l,V_{l+1},W\}$
is a partition of $V(G)$.\\[2mm]
Since $I\neq\varnothing$, no singleton set $V_i$ dominates $G$, and hence none of the sets $V_i$, $1\le i\le l+1$, is a secure dominating set. Moreover, the set $W$ does not dominate the vertex $y$, as $N_{G'}[y]\cap W=\varnothing$.
Therefore, no set in $\pi$ is a secure dominating set of $G$.\\
By arguments analogous to those used in Cases~1 and~2 above, for each $i$, $1\le i\le l+1$, the set $V_i\cup W$
is a secure dominating set of $G$. Thus, each singleton set $V_i$ is a secure coalition partner of $W$. Consequently, $\pi$ is a $sec$-partition of $G$. 
    \end{proof}
    The following result follows immediately from the proof of Theorem~\ref{mindelta2}.
\begin{corollary}
For any graph $G$ of order $n$,
$1 \le SEC(G) \le n$.
The lower bound is attained by $K_1$, while the upper bound is attained by $K_n$.
\end{corollary}
\begin{theorem}\label{theodelta}
\begin{enumerate}
\item If $G\ncong K_n$ is a graph with no isolated vertices, then
$SEC(G)\ge \delta(G)+2$.
\item Let $G\ncong \overline{K_n}$ be a graph with a nonempty set $I$ of isolated vertices, and let $G'=G[V(G)\setminus I]$.
Then $SEC(G)\ge \delta(G')+2$.
\end{enumerate}
\end{theorem}
\begin{proof}
The proof follows directly from the constructions given in the proof of Theorem~\ref{mindelta2}.
\begin{enumerate}
\item The partition $\pi=\left\{V_1,V_2,\ldots,V_{\delta+1},W\right\}$
constructed in  Theorem~\ref{mindelta2} is a secure coalition partition of $G$. Therefore, $SEC(G)\ge |\pi| =(\delta(G)+1)+1 =\delta(G)+2.$
\item The partition $\pi=\left\{V_1,V_2,\ldots,V_l,V_{l+1},W\right\}$, where $l=\delta(G')$, constructed for graphs with isolated vertices in the proof of Theorem~\ref{mindelta2}, is a secure coalition partition of $G$. Hence, $SEC(G)\ge |\pi|=l+2 =\delta(G')+2$.
\end{enumerate}
\end{proof}
\begin{proposition}\label{knoverline}
If $G\not\cong K_n$ and $G\not\cong \overline{K_n}$, then
$SEC(G)\ge 3$.
\end{proposition}
\begin{proof}
If $G$ has no isolated vertices, then $\delta(G)\ge 1$. Since $G\not\cong K_n$, Theorem~\ref{theodelta}(1) yields
$SEC(G)\ge \delta(G)+2\ge 3$. Now suppose that  $G\not\cong \overline{K_n}$ and has at least one isolated vertex. Then the graph $G'=G[V(G)\setminus I]$
is nonempty and has no isolated vertices. Consequently,
$\delta(G')\ge 1$. By Theorem~\ref{theodelta}(2),
$SEC(G)\ge \delta(G')+2\ge 3$. Therefore,  $SEC(G)\ge 3$.
\end{proof}
\begin{theorem}\label{12cor}
Let $G$ be a graph of order $n$. Then:
\begin{enumerate}
\item $SEC(G)=1$ if and only if $G\cong K_1$.
\item $SEC(G)=2$ if and only if $G\cong K_2$ or $G\cong \overline{K_n}$ for $n\ge 2$.
\end{enumerate}
\end{theorem}
\begin{proof}
By Proposition~\ref{knoverline}, if
$G\not\cong K_n$ and $G\not\cong \overline{K_n}$,
then $SEC(G)\ge 3$. Therefore, any graph $G$ satisfying $SEC(G)\le 2$ must be either a complete graph or totally disconnected graph. Since $SEC(K_n)=n$, it follows that $SEC(K_n)=1$ if and only if $n=1$. Similarly, $SEC(K_n)=2$ if and only if $n=2$. Moreover, for every $n\ge 2$, the graph $\overline{K_n}$ admits the secure coalition partition
$\bigl\{\{v_1\},\,V(\overline{K_n})\setminus\{v_1\}\bigr\}$,
and therefore $SEC(\overline{K_n})=2$.
Consequently, $SEC(G)=2$ if and only if $G\cong K_2$ or $G\cong \overline{K_n}$ for $n\ge 2$.
\end{proof}
\begin{theorem}
Let $G$ be a graph of order $n$, with maximum degree $\Delta(G)$ and domination number $\gamma(G)$. If $\pi$ is a $sec$-partition of $G$ and $Z\in \pi$, then the number of sets in $\pi$ that form a secure coalition with $Z$ is at most $\max\left\{\Delta(G)+1, n-\gamma(G)\right\}$.
\end{theorem}
\begin{proof}
Let $Z\in \pi$. Suppose that $Z$ is not a dominating set of $G$. Then there exists a vertex $x\in V(G)$ that is not dominated by the vertices of $Z$. If a set $Y\in \pi$ forms a secure coalition with $Z$, then $Y\cup Z$ must be a secure dominating set of $G$. In particular, $Y\cup Z$ must dominate $x$. Since $Z$ does not dominate $x$, the set $Y$ must contain a vertex of $N[x]$. Therefore, there can be at most $|N[x]|=d(x)+1\le \Delta(G)+1$ sets of $\pi$ that form a secure coalition with $Z$.\par
Now suppose that $Z$ is a dominating set of $G$. Then
$|Z|\ge \gamma(G)$.
Therefore, $Z$ can form a secure coalition with at most $n-\gamma(G)$ sets of $\pi$.
Thus, the number of sets in $\pi$ that form a secure coalition with $Z$ is at most $\max\left\{\Delta(G)+1, n-\gamma(G)\right\}$.
\end{proof}
\begin{theorem}
Let $G$ be a graph of order $n$ with secure domination number $\gamma_s(G)$. Then $SEC(G)\le n-\gamma_s(G)+2$.
Moreover, the bound is sharp.
\end{theorem}
\begin{proof}
If $\gamma_s(G)\le 2$, then $SEC(G)\le n\le n-\gamma_s(G)+2$,
and the result follows.
Now suppose that $\gamma_s(G)\ge 3$. Let $r=SEC(G)$,
and let $\pi=\left\{V_1,V_2,\ldots,V_r\right\}$ be a $sec$-partition of $G$. Since $\pi$ is a partition of $V(G)$, $\sum_{i=1}^{r}|V_i|=n$. Without loss of generality, assume that $V_1$ and $V_2$ form a secure coalition. Then $V_1\cup V_2$ is a secure dominating set of $G$. Hence, $|V_1|+|V_2|=|V_1\cup V_2| \ge \gamma_s(G)$ and $|V_i|\ge 1$, $3\le i\le r$. Therefore, $n=\sum_{i=1}^{r}|V_i| \ge \gamma_s(G)+(r-2)$.
Thus, $SEC(G)\le n-\gamma_s(G)+2$.
\end{proof}
The bound is sharp. For the path $P_3$, we have
$\gamma_s(P_3)=2$ and $SEC(P_3)=3$.
Hence, $SEC(P_3)=3=3-2+2=n-\gamma_s(P_3)+2$, showing that equality holds.
\section{Graphs with $SEC(G)=n$}

If $G\cong K_n$, then $SEC(G)=n$.

We now define four families of graphs, denoted by $\mathcal{F}_1$, $\mathcal{F}_2$, $\mathcal{F}_3$, and $\mathcal{F}_4$.

Let $G$ be a graph of order $n$, and let $v\in V(G)$ be a vertex of minimum degree $\delta(G)$. Define $P=N(v)$ and $Q=V(G)\setminus N[v]$. Then $|P|=\delta(G)$ and $|Q|=n-\delta(G)-1$.

Throughout, we assume that $G[Q]$ is a complete graph. Based on the structure of $G[P]$ and the interconnections between the sets $P$ and $Q$, we define the following families of graphs.
 \begin{figure}[H]
  ~~~~~~~~~~~ \includegraphics[width=0.7\textwidth]{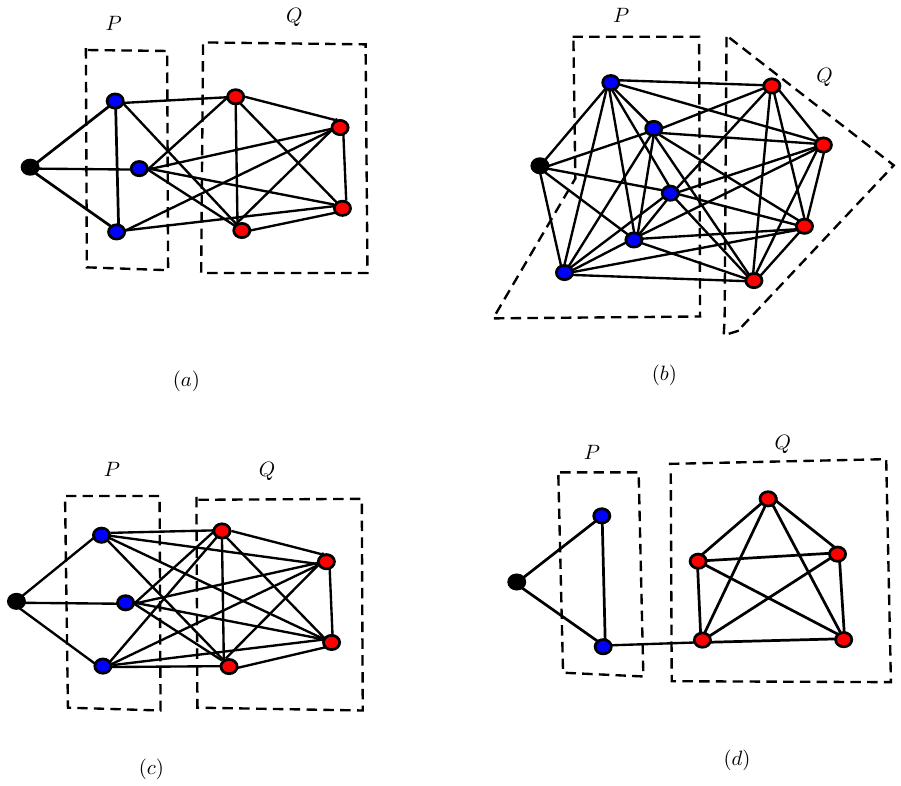}
    \caption{Examples for the graphs in the family $\mathcal{F}_1$, $\mathcal{F}_2$, $\mathcal{F}_3$ and $\mathcal{F}_4$.}
    \label{pq}
    \end{figure}
    \begin{itemize}
\item \textbf{Family $\mathcal{F}_1$.}
The graph $G[P]$ is not complete. If a vertex $x\in P$ is nonadjacent to at least one vertex of $Q$, then every vertex of $P$ that is nonadjacent to $x$ is adjacent to every vertex of $Q$. Furthermore, no vertex of $P$ is nonadjacent to two mutually nonadjacent vertices of $P$. Arbitrary edges may be added between $P$ and $Q$, including the possibility of adding none, provided that these conditions and the minimum-degree condition are preserved. See Figure~\ref{pq}(a).
\item\textbf{Family $\mathcal{F}_2$.}
The graph $G[P]$ is not complete. Every vertex of $P$ that is nonadjacent to at least one vertex of $Q$ is adjacent to every other vertex of $P$. Additional edges may be added between nonadjacent pairs of vertices in $P$, including the possibility of adding none, provided that $G[P]$ remains non-complete. See Figure~\ref{pq}(b).
\item \textbf{Family $\mathcal{F}_3$.}
The graph $G[P]$ is not complete, and every vertex of $P$ is adjacent to every vertex of $Q$. See Figure~\ref{pq}(c).
\item \textbf{Family $\mathcal{F}_4$.}
The graph $G[P]$ is complete, and at least one vertex of $P$ is adjacent to at least one vertex of $Q$. See Figure~\ref{pq}(d).
\end{itemize}
 \begin{theorem}\label{nsecnumber}
    For any connected graph $G$, $SEC(G)=n$ if and only if $G\in\left\{\mathcal{F}_1,\mathcal{F}_2,\mathcal{F}_3,\mathcal{F}_4\right\}$.
    \end{theorem}
    \begin{proof}
Let $G$ be a graph of order $n$ with $SEC(G)=n$. Then
$\pi=\{\{v\},\{v_1\},\ldots,\{v_{n-1}\}\}$
is a $sec$-partition of $G$. Let $v$ be a vertex of minimum degree, that is $d(v)=\delta(G)$. Define $P=N(v)=\{v_1,v_2,\ldots,v_\delta\}$ and
$Q=V(G)\setminus N[v]=\{v_{\delta+1},\ldots,v_{n-1}\}$.

Since $v$ is not adjacent to any vertex of $Q$, the singleton set $\{v\}$ is not a dominating set. As $\pi$ is a $sec$-partition, $\{v\}$ must form a secure coalition with some singleton set. Suppose that $\{v\}$ forms a secure coalition with $\{v_1\}$. Then $\{v,v_1\}$ is a secure dominating set of $G$.

Let $v_j\in Q$. Since $\{v,v_1\}$ is a secure dominating set, $v_1$ must dominate $v_j$. Hence, $(\{v,v_1\}\setminus\{v_1\})\cup\{v_j\}=\{v,v_j\}$ is a dominating set of $G$. Since $v$ is not adjacent to any vertex of $Q$, it follows that $v_j$ must dominate every vertex of $Q$. As $v_j$ was chosen arbitrarily, every vertex of $Q$ is adjacent to every other vertex of $Q$.\\
Similarly, if $\left\{v\right\}$ forms a secure coalition with any singleton set corresponding to the vertices of $Q$, then to form a secure coalition, all vertices of $Q$ must be adjacent to each other. Therefore, $G[Q]$ is a complete subgraph.  We now consider the following cases.\\[2mm]
 \textbf{Case 1.} $G[P]$ is a complete subgraph.\\
Since $G$ is connected and $Q\neq \varnothing$, at least one vertex of $P$ is adjacent to a vertex of $Q$. As $G[P]$ is a complete subgraph, every vertex of $P$ is adjacent to every other vertex of $P$. Moreover, $G[Q]$ is a complete subgraph. Therefore, for every $V_i=\{v_i\}$, $1\le i\le \delta$, and every $V_j=\{v_j\}$, $\delta+1\le j\le n-1$, the set $V_i\cup V_j$ is a secure dominating set of $G$. Hence, each singleton set corresponding to a vertex of $P$ forms a secure coalition with every singleton set corresponding to a vertex of $Q$.

Consequently, every singleton set in $\pi$ has a secure coalition partner, and hence $\pi$ is a secure coalition partition of $G$. Therefore, $SEC(G)=n$, and thus $G\in\mathcal{F}_4$.\\[2mm]
\textbf{Case 2.} $G[P]$ is not a complete subgraph, and every vertex of $P$ is adjacent to every vertex of $Q$.\\
In this case, for every $V_i=\{v_i\}$, $1\le i\le \delta$, and every $V_j=\{v_j\}$, $\delta+1\le j\le n-1$, the set $V_i\cup V_j$ is a secure dominating set of $G$. Hence, each singleton set corresponding to a vertex of $P$ forms a secure coalition with every singleton set corresponding to a vertex of $Q$.\\
Consequently, every singleton set in $\pi$ has a secure coalition partner, and hence $\pi$ is a secure coalition partition of $G$. Therefore, $SEC(G)=n$, and thus $G\in\mathcal{F}_3$.

\medskip

Before proceeding to the next case, we justify the restriction to graphs with $\delta(G)\ge 3$. Suppose that $\delta(G)=2$ and that $G[P]$ is not a complete subgraph. Let $P=\left\{v_{1},v_{2}\right\}$, where $v_{1}v_{2}\notin E(G)$. If there exists a vertex $u\in Q$ such that $v_{1}u\notin E(G)$, then the singleton set $\left\{v_2\right\}$ cannot form a secure coalition with any set in $\pi$. Indeed, neither $\left\{v,v_2\right\}$ nor $\left\{v_2,u\right\}$ can be a secure dominating set of $G$. This contradicts the assumption that $\pi$ is a $sec$-partition.

Therefore, when $\delta(G)=2$, every vertex of $P$ must be adjacent to every vertex of $Q$ whenever $G[P]$ is not a complete subgraph. Hence, either $G[P]$ is a complete subgraph or every vertex of $P$ is adjacent to every vertex of $Q$. Consequently, $G$ belongs to either $\mathcal{F}_3$ or $\mathcal{F}_4$.\\[2mm]
\textbf{Case 3.} $G[P]$ is not a complete subgraph, and at least one vertex of $P$ is nonadjacent to at least one vertex of $Q$.

We claim that if $\delta(G)\ge 3$ and  every vertex of $P$ is non-adjacent to every vertex of $Q$, then $SEC(G)=n$ only if $G\in\mathcal{F}_1\cup\mathcal{F}_2$.

We first consider the situation in which there exists a vertex of $P$ that is nonadjacent to at least one vertex of $P$ and at least one vertex of $Q$.

\medskip

\textbf{Subcase 3.1.} Suppose that $v_1\not\sim v_2$ and that $v_1$ is nonadjacent to a vertex $v_{q_1}\in Q$.

We show that $v_2$ must be adjacent to $v_{q_1}$. Suppose, to the contrary, that $v_2\not\sim v_{q_1}$. Consider the singleton set $\left\{v_{q_1}\right\}$. Since no vertex of $Q$ is adjacent to $v$, the set $\left\{v_{q_1}\right\}$ cannot form a secure coalition with any singleton set corresponding to a vertex of $Q$.

Since $\left\{v_{q_1},v_1\right\}$ does not dominate $v_2$. $\left\{v_{q_1}\right\}$ cannot form a secure coalition with $\left\{v_1\right\}$.  Similarly, $\left\{v_{q_1}\right\}$ cannot form a secure coalition with $\left\{v_2\right\}$, as $\left\{v_{q_1},v_2\right\}$ does not dominate $v_1$. Clearly, $\left\{v_{q_1}\right\}$ also cannot form a secure coalition with $\left\{v\right\}$.

Furthermore, if $w\in P$ is adjacent to either $v_1$ or $v_2$, then $\left\{v_{q_1},w\right\}$ fails to dominate at least one of the vertices $v_1$ and $v_2$. Hence, $\left\{v_{q_1}\right\}$ cannot form a secure coalition with any singleton set of $\pi$.
This contradicts the assumption that $\pi$ is a secure coalition partition. Therefore, $v_2$ must be adjacent to $v_{q_1}$.

Consequently, if a vertex $x\in P$ is nonadjacent to some vertex of $Q$, then every vertex of $P$ that is nonadjacent to $x$ must be adjacent to every vertex of $Q$.

\medskip

\textbf{Subcase 3.2.} Suppose that $v_1\not\sim v_2$, $v_1\not\sim v_3$, $v_2\not\sim v_3$, and that $v_1$ is nonadjacent to a vertex $v_{q_1}\in Q$.\\
By Subcase~3.1, both $v_2$ and $v_3$ must be adjacent to $v_{q_1}$. Now consider the singleton set $\left\{v_3\right\}$. Since no vertex of $Q$ is adjacent to $v$, the set $\left\{v_3\right\}$ cannot form a secure coalition with any singleton set corresponding to a vertex of $Q$. In particular, the set $\left\{v_3,v_{q_1}\right\}$ does not dominate $v_1$.

Further, $\left\{v_3,v_2\right\}$ does not dominate $v_1$, and $\left\{v_3,v_1\right\}$ does not dominate $v_2$. Hence, $\left\{v_3\right\}$ cannot form a secure coalition with $\left\{v_2\right\}$ or $\left\{v_1\right\}$. It also cannot form a secure coalition with $\left\{v\right\}$. Therefore, $\left\{v_3\right\}$ does not form a secure coalition with any singleton set of $\pi$, contradicting the assumption that $SEC(G)=n$.

It follows that no vertex of $P$ can be nonadjacent to two mutually nonadjacent vertices of $P$. Hence, by Subcases~3.1 and~3.2, $G\in\mathcal{F}_1$.\\[2mm]
Next, consider the case in which every vertex of $P$ that is nonadjacent to at least one vertex of $Q$ is adjacent to every other vertex of $P$.

In this situation, the singleton set $\left\{v\right\}$ forms a secure coalition with every singleton set corresponding to a vertex of $Q$, as well as with every singleton set corresponding to a vertex of $P$ that is adjacent to all vertices of $Q$. Moreover, each singleton set corresponding to a vertex of $P$ that is nonadjacent to at least one vertex of $Q$ forms a secure coalition with a singleton set corresponding to a vertex of $Q$. Therefore, every singleton set in $\pi$ has a secure coalition partner. Hence, $\pi$ is a secure coalition partition of $G$, and consequently $SEC(G)=n$. It follows that $G\in\mathcal{F}_2$.

\medskip

Conversely, suppose that $G\in\mathcal{F}_1$. By the definition of $\mathcal{F}_1$, every singleton set corresponding to a vertex of $Q$, as well as every singleton set corresponding to a vertex of $P$ that is adjacent to all vertices of $Q$, forms a secure coalition with $\left\{v\right\}$. Furthermore, every singleton set corresponding to a vertex of $P$ that is nonadjacent to at least one vertex of $P$ and at least one vertex of $Q$ forms a secure coalition with a singleton set corresponding to a vertex of $P$ that is adjacent to all vertices of $Q$. Thus, every singleton set in $\pi$ has a secure coalition partner, and hence
$SEC(G)=n$.

Similarly, if $G\in\mathcal{F}_2$, $\mathcal{F}_3$, or $\mathcal{F}_4$, then the existence of secure coalition partners for all singleton sets follows from the arguments given above. Therefore, $SEC(G)=n$ for every graph belonging to $\mathcal{F}_1\cup\mathcal{F}_2\cup\mathcal{F}_3\cup\mathcal{F}_4$.
\end{proof}
\begin{theorem}
Let $G$ be a disconnected graph of order $n$. Then $SEC(G)=n$
if and only if $G\cong K_p\cup K_q$, where $p+q=n$.
\end{theorem}
\begin{proof}
Suppose that $G$ is a disconnected graph with $SEC(G)=n$.\\ Then
$\pi=\bigl\{\{v_1\},\{v_2\},\ldots,\{v_n\}\bigr\}$
is a secure coalition partition of $G$.\\[2mm]
\textbf{Claim 1.} $G$ has exactly two connected components.\\
Suppose, to the contrary, that $G$ has at least three connected components. Let $u$ and $v$ be any two vertices of $G$. Since $u$ and $v$ belong to at most two components, the set $\left\{u,v\right\}$ cannot dominate the vertices of a third component. Hence, $\left\{u,v\right\}$ is not a dominating set of $G$. Consequently, no singleton set can form a secure coalition with any singleton set of $\pi$, contradicting the assumption that $\pi$ is a secure coalition partition. Therefore, $G$ has at most two components. Since $G$ is disconnected, it must have exactly two connected components.

\medskip

\textbf{Claim 2.} Each component of $G$ is a complete graph.\\
Let $\Gamma_1$ and $\Gamma_2$ be the two components of $G$. Suppose that $\Gamma_1$ is not a complete graph. Then there exists a vertex $x\in V(\Gamma_1)$ that is nonadjacent to some vertex of $\Gamma_1$.

 The singleton set $\{x\}$ cannot form a secure coalition with any singleton set corresponding to a vertex of $\Gamma_2$, as it is non- adjacent to some vertices of $\Gamma_1$. Moreover, if $y\in V(\Gamma_1)$, then the set $\left\{x,y\right\}$ fails to dominate the vertices of $\Gamma_2$. Thus, $\left\{x\right\}$ does not form a secure coalition with any singleton set of $\pi$, contradicting the fact that $\pi$ is a secure coalition partition.

Hence, $\Gamma_1$ must be complete. By the same argument, $\Gamma_2$ must also be complete. Therefore, $G\cong K_p\cup K_q$, where $p+q=n$.

Conversely, let $G\cong K_p\cup K_q$. Then every singleton set corresponding to a vertex of $K_p$ forms a secure coalition with every singleton set corresponding to a vertex of $K_q$, since the union of any such pair is a secure dominating set of $G$. Hence,
$\bigl\{\{v_1\},\{v_2\},\ldots,\{v_n\}\bigr\}$
is a secure coalition partition of $G$, and therefore $SEC(G)=n$.
\end{proof}

As a consequence of the above theorem, we obtain the following characterization of graphs with minimum degree zero and secure coalition number equal to their order.

\begin{corollary}
Let $G$ be a graph of order $n$ with $\delta(G)=0$. Then
$SEC(G)=n$ if and only if $G\cong K_1\cup K_{n-1}$.
\end{corollary}
\begin{proof}
If $\delta(G)=0$ and $SEC(G)=n$, then by the theorem,
$G\cong K_p\cup K_q$. Since $\delta(G)=0$, one of the components must be an isolated vertex. Hence, $G\cong K_1\cup K_{n-1}$.

Conversely, if $G\cong K_1\cup K_{n-1}$, then one can easily verify that, $SEC(G)=n$.
\end{proof}
\section{Secure coalition in trees}
\begin{corollary}\label{nsecnumber1}
For any tree $T$ of order $n$, $SEC(T)=n$ if and only if $T\cong P_n$, where $n\le 4$.
\end{corollary}
\begin{proof}
By Theorem~\ref{nsecnumber}, there exists no tree of order $n\ge 5$ which satisfies the condition $SEC(T)=n$. For trees of order $n=1,2,3$, the result is trivial. Now consider a tree of order $4$. Then either $T\cong S_4$ or $T\cong P_4$. Further, $SEC(S_4)=3\neq n$ and $SEC(P_4)=4=n$. Thus, $P_4$ is the only tree that satisfies the condition $SEC(T)=n$.
\end{proof}
\begin{theorem}\label{treen-2}
Let $T$ be a tree of order $n\ge 5$. If $T\not\cong P_5$, then
$SEC(T)\le n-2$.
\end{theorem}
\begin{proof}
By Corollary~\ref{nsecnumber1}, we have $SEC(T)<n$ for every tree $T$ of order $n\ge 5$. Therefore, it suffices to show that, for $n\ge 5$ and $T\not\cong P_5$, we have $SEC(T)\neq n-1$.

Suppose, to the contrary, that $SEC(T)=n-1$. Let $\pi$ be a secure coalition partition of $T$ with $|\pi|=n-1$. Then exactly one set of $\pi$ has cardinality $2$, say $A$, while remaining are singleton sets.

Let $x\in V(T)$ be a pendant vertex such that $x\notin A$.

\textbf{Claim 1.} The singleton set $\{x\}$ cannot form a secure coalition with any singleton set of $\pi$.

Suppose, to the contrary, that $\{x\}$ forms a secure coalition with a singleton set $\{z\}\in\pi$. Then $S=\{x\}\cup\{z\}$
is a secure dominating set of $T$. Since $x$ is a pendant vertex, the vertex $z$ must dominate every vertex not dominated by $x$. Moreover, because $T$ is a tree, no two vertices of $N(z)$ are adjacent; otherwise, they would form a cycle together with $z$.
Now let $w\in N(z)\setminus N[x]$. Since $w$ is dominated by $z$, the set $(S\setminus\{z\})\cup\{w\} =\{x,w\}$
must be a dominating set of $T$. However, one can easily observe that, the set $\{x,w\}$ cannot dominate all vertices of $T$. This contradicts the assumption that $S$ is a secure dominating set. Therefore, $\{x\}$ cannot form a secure coalition with any singleton set of $\pi$.

Hence, the only possible secure coalition partner of $\{x\}$ is the set $A$. Since every tree has at least two pendant vertices, let $y\neq x$ be another pendant vertex of $T$. Then $A$ must contain either $y$ or the support vertex adjacent to $y$; otherwise, $A\cup\{x\}$ fails to dominate $y$ and therefore is not a dominating set.

We now consider the following cases.

\textbf{Case 1.} The vertex $y$ belongs to $A$.

Let $A=\{a,y\}$, where $a\neq y$. Since $A\cup\{x\}$ is a dominating set, the vertex $a$ must dominate all vertices that are not dominated by $x$ and $y$.

\textbf{Subcase 1.1.} The vertex $a$ is adjacent to neither $x$ nor $y$.

\textbf{Claim 1.1.1.} The vertex $a$ is not a pendant vertex.

Suppose, to the contrary, that $a$ is a pendant vertex. Since $\{x\}\cup A=\{x,a,y\}$ is a dominating set, it follows that $T$ must be isomorphic to one of the graphs $G_1$ or $G_2$ shown in Fig.~\ref{figseccoalition}.

\begin{figure}[H]
\centering
\includegraphics[width=0.75\textwidth]{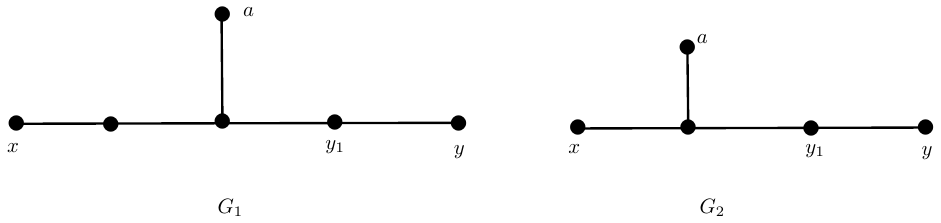}
\caption{Graphs $G_1$ and $G_2$.}
\label{figseccoalition}
\end{figure}

In both $G_1$ and $G_2$, the singleton set corresponding to the support vertex $y_1$ of $y$ cannot form a secure coalition with any part of $\pi$. This contradicts the assumption that $\pi$ is a secure coalition partition. Hence, $a$ is not a pendant vertex.

Furthermore, since $T\not\cong P_5$, every vertex adjacent to $a$ that does not belong to $N(x)\cup N(y)$ must be a pendant vertex. Let $a_1$ be such a pendant vertex adjacent to $a$.

\textbf{Claim 1.1.2.} The singleton set $\left\{a_1\right\}$ does not form a secure coalition with any set of $\pi$.

Since $a_1$ is a pendant vertex, an argument identical to that used in Claim~1 shows that $\left\{a_1\right\}$ cannot form a secure coalition with any singleton set of $\pi$. Therefore, its only possible secure coalition partner is the set $A=\left\{a,y\right\}$.

However, $S=\{a_1\}\cup A=\left\{a_1,a,y\right\}$
is not a dominating set, since none of its vertices dominates the pendant vertex $x$. This contradiction shows that $\left\{a_1\right\}$ cannot form a secure coalition with any set of $\pi$. Hence, the claim follows.

\textbf{Subcase 1.2.} The vertex $a$ is adjacent to exactly one of $x$ and $y$.

\textbf{Claim 1.2.1.} The singleton set corresponding to the support vertex $y_1$ of $y$ cannot form a secure coalition with any set of $\pi$.

First suppose that $a\sim y$ and $a\not\sim x$. Then $a=y_1$. Since $n\ge 5$, the vertex $y_1$ is adjacent to a pendant vertex other than $y$, say $a_2$. By Claim~1, the singleton set $\left\{a_2\right\}$ cannot form a secure coalition with any singleton set of $\pi$. Moreover, its only possible coalition partner is $A=\left\{y,y_1\right\}$, but $\left\{a_2,y,y_1\right\}$ is not a secure dominating set. A contradiction.

Now suppose that $a\sim x$ and $a\not\sim y$. Consider the singleton set $\left\{y_1\right\}$. Since neither $y_1$ nor any vertex other than $a$ dominates both pendant vertices $x$ and $a_2$, the set $\left\{y_1\right\}$ cannot form a secure coalition with any singleton set of $\pi$. Furthermore,
$\left\{y_1\right\}\cup A=\left\{y_1,a,y\right\}$ is not a secure dominating set. Hence, $\left\{y_1\right\}$ does not form a secure coalition with any set of $\pi$, a contradiction.

Therefore, in either case, the singleton set corresponding to the support vertex $y_1$ of $y$ cannot form a secure coalition with any part of $\pi$. This proves the claim.

\textbf{Subcase 1.3.} The vertex $a$ is adjacent to both $x$ and $y$.

Since $x$ and $y$ are pendant vertices and every neighbor of $a$ other than those in $N(x)\cup N(y)$ is also a pendant vertex, it follows that $T\cong S_n$.
Observe that $\left\{x\right\}\cup A=\left\{x,a,y\right\}$
is a dominating set of $T$. Since $n\ge 5$, there exists a pendant vertex $a_1\neq x,y$ adjacent to $a$. As $a_1$ is dominated only by $a$, replacing $a$ with $a_1$ yields
$(\left\{x,a,y\right\}\setminus\left\{a\right\})\cup\left\{a_1\right\}=\left\{x,y,a_1\right\}$, which is not a dominating set of $T$. Therefore, $\left\{x,a,y\right\}$ is not a secure dominating set. Consequently, $\left\{x\right\}$ and $A$ do not form a secure coalition, a contradiction.

\textbf{Case 2.} The support vertex $y_1$ of $y$ belongs to $A$, and there exists a vertex $a\neq y$ such that $A=\{y_1,a\}$.

\textbf{Subcase 2.1.} The vertex $x$ is adjacent to $y_1$.

\textbf{Claim 2.1.1.} Every vertex adjacent to $y_1$ is a pendant vertex.

Suppose, to the contrary, that there exists a vertex $y_1'\in N(y_1)$ with $d(y_1')\ge 2$. Since $T$ is a tree, $y_1'$ is adjacent to a pendant vertex, say $y_2'$.

First assume that $a=y_1'$. Then the singleton set $\{y_2'\}$ cannot form a secure coalition with any singleton set of $\pi$. Moreover,
$\left\{y_2'\right\}\cup\left\{y_1,y_1'\right\}$
is a dominating set but not a secure dominating set, yielding a contradiction.

Next assume that $a=y_2'$. By a similar argument,
$\left\{y_1'\right\}\cup\left\{y_1,y_2'\right\}$
is a dominating set but not a secure dominating set, again a contradiction.

Finally, suppose that $a\notin\left\{y_1',y_2'\right\}$. Since $a$ cannot be adjacent to both $x$ and $y$, it follows that
$\left\{y_1'\right\}\cup\left\{y_1,a\right\}$ is not a secure dominating set, a contradiction. Hence, every neighbor of $y_1$ is a pendant vertex. Consequently, $T$ is a star, that is,
$T\cong S_n$. However, for $n\ge 5$,
$SEC(S_n)=3<n-1$, contradicting the assumption that $SEC(T)=n-1$. 

\textbf{Subcase 2.2.} The vertex $x$ is not adjacent to $y_1$.

\textbf{Claim 2.2.1.} The vertex $a$ is adjacent to $x$.

Suppose, to the contrary, that $a\not\sim x$. Then the singleton set $\{y\}$ cannot form a secure coalition with any singleton set of $\pi$. Moreover, $\{y\}\cup\{a,y_1\}$
fails to be a dominating set. This contradiction proves that $a\sim x$.

\textbf{Claim 2.2.2.} The vertex $y_1$ is not adjacent to any pendant vertex other than $y$.

Suppose that $y_1$ is adjacent to a pendant vertex $y_3\neq y$. Then $\{x\}\cup\{y_1,a\}$
is not a secure dominating set, contradicting the fact that $\{x\}$ and $A=\{y_1,a\}$ must form a secure coalition. Hence, $y_1$ is not adjacent to any pendant vertex other than $y$.

\textbf{Claim 2.2.3.} The vertex $a$ is not adjacent to any pendant vertex other than $x$.

Suppose that $a$ is adjacent to a pendant vertex $x_2\neq x$. Then the singleton set $\{y\}$ cannot form a secure coalition with any set of $\pi$, a contradiction. Therefore, $a$ is not adjacent to any pendant vertex other than $x$.

\textbf{Claim 2.2.4.} The vertices $y_1$ and $a$ are adjacent.

Suppose that $y_1\not\sim a$. Then the singleton set $\{y\}$ cannot form a secure coalition with any set of $\pi$, contradicting the assumption that $\pi$ is a secure coalition partition. Hence, $y_1\sim a$.

Furthermore, one can verify that if either $a$ or $y_1$ is adjacent to a non-pendant vertex, then the resulting tree cannot satisfy $SEC(T)=n-1$. Therefore, this subcase also leads to a contradiction.\\
Furthermore, Claims~2.2.1--2.2.4 imply that every neighbor of $a$ other than $x$ and every neighbor of $y_1$ other than $y$ is a pendant vertex. Consequently, the resulting structure does not admit a secure coalition partition of cardinality $n-1$, contradicting the assumption that $SEC(T)=n-1$.\\[2mm]
We now consider the cases in which $x\in A$.

\textbf{Case 3.} The vertex $y\notin A$.

This case is analogous to Case~1. Indeed, by interchanging the roles of $x$ and $y$ in the arguments of Case~1, we obtain the  contradiction. Hence, this case cannot occur.

\textbf{Case 4.} The vertex $y\in A$.

Since $n\ge 5$ and $T\not\cong P_5$, there exists a vertex $z\in V(T)$ that is not dominated by either $x$ or $y$. Consequently, every pendant vertex of $T$ distinct from $x$ and $y$ must be dominated by $z$; otherwise, \{x,y,z\} would fail to be a dominating set, and the singleton set $\{z\}$ would not be able to form a secure coalition with any singleton part of $\pi$.

Let $x_1$ and $y_1$ denote the support vertices of $x$ and $y$, respectively. Since $\{x,y,z\}$ must be a dominating set, the vertex $z$ must be adjacent to both $x_1$ and $y_1$. Indeed, if $z\not\sim y_1$, then $\{x,y\}\cup\{y_1\}$ is not a dominating set. A similar argument applies to $x_1$.

Observe that $x_1$ and $y_1$ can dominate only the vertices $x$, $y$, and $z$. If either $x_1$ or $y_1$ dominates an additional vertex $w$, then $\{x,y\}\cup\{w\}$ is not a dominating set, and the singleton set $\{w\}$ cannot form a secure coalition with any singleton part of $\pi$. Therefore, $d(x_1)=d(y_1)=2$. It follows that every vertex in $V(T)\setminus\{x,y,x_1,y_1\}$
must be dominated by $z$. Since $T\not\cong P_5$, the vertex $z$ has a neighbor other than $x_1$ and $y_1$. One can now verify that neither $\{x_1\}$ nor $\{y_1\}$ can form a secure coalition with any set of $\pi$, contradicting the assumption that $\pi$ is a secure coalition partition.

Therefore, for every tree $T$ of order $n\ge 5$ with $T\not\cong P_5$, we have $SEC(T)\neq n-1$. Thus, it follows that
$SEC(T)\le n-2$.
\end{proof}
As an immediate consequence of Theorem~\ref{treen-2} and Corollary~\ref{nsecnumber1}, we obtain the following characterization.

\begin{corollary}
For a tree $T$ of order $n$,
$SEC(T)=n-1$ if and only if $T\cong S_4$ or $T\cong P_5$.
\end{corollary}
\section{Secure coalition graphs}
    A graph which is naturally associated with a $sec$-partition $\pi$ of a graph $G$ defined as follows.
\begin{definition}
    Let $G$ be a graph with $sec$-partition $\pi=\left\{V_1,V_2,\dots,V_k\right\}$. The secure coalition graph $SCG(G,\pi)$ of $G$ is the graph with vertex set $V_1,V_2,\dots,V_k$ and two vertices $V_i$ and $V_j$ are adjacent in $SCG(G,\pi)$ if and only if the sets $V_i$ and $V_j$ are secure coalition partners in $\pi$, that is, neither $V_i$ nor $V_j$ is a secure dominating set of $G$, but $V_i\cup V_j$ is a secure dominating set of $G$.\end{definition}
     The complete graph $K_n$ has exactly one $sec$-partition, namely, its singleton partition, for which $SCG(K_n,\pi)\cong \overline{K_n}$. The singleton coalition partition of cycle $C_5=(v_1,v_2,v_3,v_4,v_5,v_1)$ gives $SCG(C_5,\pi_1)\cong 2P_3$, where $2P_3$ is a disjoint union of two paths of order $3$. The $sec$-partition $\pi_2=\left\{\left\{v_1,v_5\right\},\left\{v_3\right\},\left\{v_2\right\},\left\{v_4\right\}\right\}$ of $C_5$ results in $SCG(C_5,\pi_2)\cong S_4+e$, where $S_4+e$ is the graph obtained by star on $4$ vertices by adding an edge between two non-adjacent pair of vertices. \par
    Since every graph $G$ has at least one $sec$-partition $\pi$, every graph has at least one associated secure coalition graph $H = SCG(G,\pi)$ and can have many associated secure coalition
graphs, depending on the number of $sec$-partitions that it has. In the next theorem, we show that, every graph $G$ with $\delta\ge 1$ is a secure coalition graph.
\begin{theorem}\label{spi}
For every graph $G$ with no isolated vertices, there is a graph $H$ and some $sec$-partition $\pi$ of $H$, such that $SCG(H,\pi)\cong G$.
\end{theorem}
\begin{proof}
 Let $G$ be a graph with vertex set $
V(G) = \{v_1, v_2, \dots, v_n\}$,
$|V(G)|=n$. Let 
$E(G) = \{e_1, e_2, \dots, e_m\}$ and  
$E(\overline{G}) = \{q_1, q_2, \dots, q_{\overline{m}}\}$,
where $m + \overline{m} = \binom{n}{2}$.
We aim to construct a graph $H$ and a secure coalition partition $\pi$ of $H$ such that 
$SCG(H,\pi) \cong G$.\par
We begin the construction of $H$ with the complete graph $K_n$, whose vertices are labeled 
$\{v_1, v_2, \dots, v_n\}$, corresponding to the vertices of $G$. These vertices are referred to as the base vertices of $H$. 
The partition $\pi$ is initially taken as the singleton partition 
$\pi = \{V_1, V_2, \dots, V_n\}$, where each set contains exactly one base vertex, i.e., $V_i = \{v_i\}$ for $i = 1,2,\dots,n$. As the construction progresses, more vertices will be incorporated into $H$ and, subsequently, into the partition sets $V_i$.\par
For each vertex $v_i \in V(G)$, where $1 \le i \le n$, we extend the graph $H$ as follows.  
A new vertex $u_i$ is introduced and joined to all base vertices of $H$ except $v_i$.  
This new vertex $u_i$ is then included in the set $V_i$ of the partition $\pi$.\par
Next, for each edge $q_i = v_jv_k \in E(\overline{G})$, we proceed as follows.  
Introduce a new vertex, denoted $y_{jk}$, and join it to all base vertices of $H$ except $v_j$ and $v_k$.  
Furthermore, connect $y_{jk}$ to all vertices $u_i$ ($1\le i \le n)$ except $u_j$ and $u_k$.  
Finally, place $y_{jk}$ into any set of the partition $\pi$ other than $V_j$ or $V_k$.  Here it is important to note that, if $G$ has a full vertex, say $v_1$, and there are $\overline{m}$ vertices $y_{jk}$ ($1\le j \neq k\le \overline{m}$) corresponding to edges of $\overline{G}$, then not all of these $\overline{m}$ vertices can be placed in $V_1$.\par
It is important to note that none of the sets $V_i$, $1\le i\le n$, is a dominating set. Indeed, the vertices $u_i$ and $v_i$ fail to dominate the vertex corresponding to an edge of $\overline{G}$ incident with the vertex $i$.\\[2mm]
\textbf{Claim 1}. $SCG(H,\pi)\cong G$.\\[2mm]
\textbf{Claim 1.1}. If $v_jv_k \in E(G)$, then the set $S = V_j \cup V_k$ forms a secure coalition in $\pi$.\\[2mm]
Indeed, the base vertices $v_j$ and $v_k$ together dominate all vertices of $H$. Moreover, if $v_j$ is removed, the vertex $u_j \in S$ dominates all vertices not dominated by $v_k$ except $u_k$, and $u_k \in S$. Similarly, if $v_k$ is removed, the vertex $u_k \in S$ dominates all vertices not dominated by $v_j$ except $u_j$, and $u_j \in S$. Thus $V_j\cup V_k$ forms secure coalition in $\pi$. The  claim holds.\\[2mm]
\textbf{Claim 1.2}. If $v_jv_k \notin E(G)$, then the sets $V_j$ and $V_k$ do not form a secure coalition in $\pi$.\\[2mm]
 Since the vertex $y_{jk}\in H$ is not adjacent to $v_j$, $v_k$, $u_j$, and $u_k$, the set $V_j \cup V_k$ fails to be a dominating set. Thus $V_j\cup V_k$ do not form secure coalition in $\pi$. The claim holds.\par
 By Claim~1.1 and 1.2, the graph generated by $\pi$ is of order $n$. The vertices corresponding to the sets $V_j$ and $V_k$ of $\pi$ in $SCG(H,\pi)$ are adjacent if $v_jv_k\in E(G)$ and non-adjacent if $v_jv_k\not \in E(G)$. Thus, $SCG(H,\pi)\cong G$.
\end{proof}
Notice that for a given graph $G$ of order $n$ and size $m$, where $m+\overline{m}=\binom{n}{2}$, the graph $H$ constructed in Theorem~\ref{spi} has: order $n(H)=2n+\overline{m}$, size $m(H)=\binom{n}{2}+n(n-1)+2\overline{m}(n-2)$.
\begin{theorem}
If $G$ is a graph with $\delta(G)=0$ and has at least one edge, then there exist no graph $H$ and $sec$-partition $\pi$ of $H$ such that $SCG(H,\pi) \cong G$.
\end{theorem}
\begin{proof}
By the definition of a $sec$-partition, a set fails to form a secure coalition with any other set if and only if it is a secure dominating set consisting of a single vertex of degree $n-1$. This situation occurs only in a complete graph, where every vertex is a secure dominating set.\par Assume that $\delta(G) = 0$ and $G$ has at least one edge. Suppose, for contradiction, that there exists a graph $H$ with a $sec$-partition $\pi$ such that
$SCG(H,\pi) \cong G$. Since $\delta(G) = 0$, the graph $G$ has an isolated vertex. Let the corresponding part in $\pi$ be $V_k$. Then $V_k$ does not form a secure coalition with any other part of $\pi$. By the definition of a $sec$-partition, $V_k = \{v\}$, where $d(v)=n-1$. That is, $v$ is adjacent to every vertex of $H$. Now, since $G$ has at least one edge, there exist two adjacent vertices in $SCG(H,\pi)$. Hence, there exist parts $V_i, V_j \in \pi$ such that neither $V_i$ nor $V_j$ is a secure dominating set individually, but $V_i \cup V_j$ forms a secure dominating set. However, since $v$ is adjacent to every vertex of $H$, for any vertex $u \in V_i$, by the definition of secure domination each $\{u\}$ must be a secure dominating set in $H$, which contradicts the existence of parts $V_i$ and $V_j$ that are not secure dominating sets individually.
Therefore, no such graph $H$ exists, and hence there is no graph $H$ such that $SCG(H,\pi) \cong G$.
\end{proof}
    \section*{Conclusion}
    In this paper, we first prove that every graph admits a secure coalition partition. We then completely characterize the graphs for which $SEC(G) \in \{1,2,n\}$ and establish an upper bound on $SEC(G)$ in terms of the order of the graph and its secure domination number. Furthermore, we characterize all trees with $SEC(T) = n-1$. Finally, we demonstrate that every graph $G$ without isolated vertices is a secure coalition graph.
 
    \end{document}